\documentclass[11pt]{amsart}
\usepackage[margin=1in]{geometry}
\usepackage{graphicx}
\usepackage{comment}
\usepackage{float}
\usepackage{amssymb,amsmath}
\usepackage{epsfig,subfigure}
\usepackage{epstopdf}
\usepackage{enumerate}
\usepackage[title]{appendix}
\usepackage[dvipsnames]{xcolor}

\usepackage[colorlinks=true, pdfstartview=FitV, linkcolor=RoyalBlue,citecolor=ForestGreen, urlcolor=blue]{hyperref}
\def\CC{{\mathbb{C}}}

\newtheorem{theorem}{Theorem}[section]

\newtheorem{lemma}[theorem]{Lemma}
\newtheorem{proposition}[theorem]{Proposition}

\newtheorem{remark}[theorem]{Remark}

\title{Dynamics of rotationally invariant polynomial root sets  under iterated differentiations}
\author{Andr\'e Galligo}
\address{Universit\'e C\^ote d'Azur,
Laboratoire de Math\'ematiques
"J.A. Dieudonn\'e", UMR CNRS 7351, Inria, 
Nice, France. }
\email{andre.galligo@univ-cotedazur.fr}
\author{Joseph Najnudel}
\address{University of Bristol, School of Mathematics, Bristol, United Kingdom.}
\email{joseph.najnudel@bristol.ac.uk}
\author{Truong Vu} 
\address{University of Illinois at Chicago, Department of Mathematics, Statistics, and Computer Science, Chicago, USA.}
\email{tvu25@uic.edu}
\date{}
\begin{document}

\maketitle

\begin{abstract}
We associate to  an $N$-sample of a given rotationally invariant probability measure $\mu_0$ with compact support in the complex plane,  a polynomial $P_N$ with roots given by the sample. Then, for $t \in (0,1)$, we consider the empirical measure  $\mu_t^{N}$ associated to the root set of the $\lfloor t N\rfloor$-th derivative of  $P_N$. 
 A question posed by O'Rourke and Steinerberger \cite{o2021nonlocal}, reformulated as a conjecture by Hoskins and Kabluchko \cite{hoskins2021dynamics}, and recently reaffirmed by Campbell, O'Rourke and Renfrew \cite{campbell2024fractional}, states that under suitable conditions of regularity on $\mu_0$, for an i.i.d. sample, 
 $\mu_t^{N}$ converges to a rotationally invariant probability measure $\mu_t$  when $N$ tends to infinity, and that $(1-t)\mu_t$ has a radial density 
 $x \mapsto \psi(x,t)$ satisfying the following partial differential equation: 
 
	\begin{equation} \label{PDErotational}
		\frac{ \partial \psi(x,t) }{\partial t} = \frac{ \partial}{\partial x} \left( \frac{ \psi(x,t) }{ \frac{1}{x} \int_0^x \psi(y,t) dy } \right),
	\end{equation}
In \cite{hoskins2021dynamics}, this equation is reformulated as an equation 
on the distribution function $\Psi_t$ of the radial part of $(1-t) \mu_t$: 
\begin{equation}
\label{equationPsixtabstract}
\frac{\partial \Psi_t (x)}{\partial t} 
= x \frac{\frac{\partial \Psi_t (x)}{\partial x}  } {\Psi_t(x)} - 1. 
\end{equation}
Restricting our study to a specific family of $N$-samplings, we are able to prove a variant of the conjecture above. 
We also emphasize the important differences between the two-dimensional setting and the one-dimensional setting, illustrated in our Theorem \ref{monotonicity}.

\end{abstract}

\section{Introduction}\label{intro}

Let  $\mu_0$ be a  probability distribution on the complex plane $\mathbb C$ supported in a compact set. Then, for each integer $N \geq 1$, we consider  complex-valued (possibly) random variables $z_1,z_2,\ldots, z_N$ distributed  according to a sampling of a deterministic measure $\mu_0$: more precisely, we assume
$$ \frac{1}{N} \sum_{j=1}^N \delta_{z_j} \underset{N \rightarrow \infty}{\longrightarrow} \mu_0 \quad \text{in probability}$$
 where $\delta_z$ denotes the Dirac measure at $z\in \mathbb C$, and where both sides are viewed as random elements with values in the space of finite measures on $\mathbb C$ endowed with the weak convergence topology. 
Let $P_N$ be the monic polynomial of degree $N$ whose roots are $z_1,\ldots,z_N$, that is
$$
P_N(z) := \prod_{k=1}^N (z-z_k), \qquad z\in \mathbb C.
$$
The critical points of $P_N$ are defined as the roots of its derivative $P_N'$. It is known from ~\cite{Kabluchko2015} (first conjectured by Pemantle and Rivin~\cite{pemantle_rivin}) that for an i.i.d. sampling of $\mu_0$, the critical points of $P_N$ have the same asymptotic distribution as the roots of $P_N$.  More precisely, we have
$$
\frac{1}{N-1} \sum_{z\in \mathbb C: P_N'(z) =0} \delta_z \rightarrow \mu_0\quad \text{in probability}.
$$

We are interested in the asymptotic distribution, as $N\to\infty$, of the roots of the 
$k$-th derivative of $P_N$, denoted by $P_N^{(k)}$. It was proven that the previous behavior (same distribution) extends to higher derivatives when $k$ is finite, see e.g. \cite{byun-lee-reddy}. But, when $k$ also tends to infinity as a function of $N$, there are different regimes. A  regime which is often considered is
$k= \lfloor t N\rfloor$, with a fixed  $t \in (0,1)$, $ \lfloor x \rfloor$ denoting the greatest  integer less than  $x$. We refer to \cite{galligo2024anti}, \cite{galligo2025dynamics},\cite{michelen2024almost},\cite{michelen2024zeros}, \cite{Feng_Yao2019zeros},\cite{hoskins2021dynamics}, \cite{jalowy2025zeros} and references therein for such developments. Connections between this setting, combinatorics, and free probability are provided in \cite{arizmendi-fujie-ueda} and \cite{arizmendi2024finitefreecumulantsmultiplicative}. 

If $z_1, \dots, z_N$ are real, and then $\mu_0$ supported in $\mathbb{R}$, it has been proven that the empirical measure of the roots of 
$P_N^{(\lfloor tN \rfloor)}$ converges in probability to a measure $\mu_t$ depending only on $\mu_0$ and $t$:
$$
\frac{1}{N (1-t)} \sum_{z\in \CC : P_N^{(\lfloor t N\rfloor)}(z)=0} \delta_z \;\; \rightarrow \mu_t
\quad \text{in probability}.
$$
Moreover, $\mu_t$ can be expressed in terms of solutions of partial differential equations, called Steinerberger PDE's: see \cite{alazard-lazar-nguyen} and \cite{kiselev-tan}.   There is a similar study when $\mu_0$ is supported on the unit circle and the differentiation is considered with respect to the argument: see \cite{kabluchko2021repeated}.  

Very few results are known in the general case where the support $\mu_0$ is $2$-dimensional: a discussion on this problem is given in \cite{Galligo}. A similar conjecture as in the one-dimensional setting has been stated when $z_1, \dots, z_N$ are i.i.d. random variables distributed according to a measure $\mu_0$ which is rotationally invariant: see \cite{hoskins2021dynamics}, \cite{steinerberger2019nonlocal}, \cite{o2021nonlocal}, \cite{hoskins2022semicircle}. In that case, it is conjectured that the empirical measure of the roots of $P_N^{(\lfloor tN \rfloor)}$ converges in probability to a measure $\mu_t$ satisfying 
the PDE's \eqref{PDErotational} and \eqref{equationPsixtabstract}, under suitable regularity conditions.

In \cite{hoskins2021dynamics}, it was predicted that a simple way to express $\mu_t$ through $\mu_0$ is stated as follows. Let the distribution function of the radial part of $(1-t) \mu_t$ be
	$\Psi_t(x) := \Psi(x,t) := \int_0^x \psi(y, t) dy $
	at time $t \in (0,1)$. Then, there is a constant loss of mass for the solution: 
			$ \frac{d}{dt} \int_0^\infty \psi(x,t) d x = -1 $ and
\begin{equation} \label{Psi} \frac{ \Psi_t^{\langle -1 \rangle}(x)}{x}  =  \frac{ \Psi_0^{\langle -1 \rangle}(x+t)}{x+t} 
\end{equation}
	for $0 < x < 1-t$ and $0 <  t < 1 $, where  $( \cdot )^{\langle -1 \rangle}$ denotes inversion with respect to composition.

The heuristics behind this conjecture has been inspired by a mean field strategy similar to the one used in the one-dimensional case: see O’Rourke-Steinerberger \cite{o2021nonlocal}, Hoskins-Kabluchko \cite{hoskins2021dynamics}. The assumption 
that $z_1, \dots, z_N$ are i.i.d. does not look very natural since it is not stable by differentiation: it is likely that the conjecture is satisfied under more generic assumptions. Indeed, Campbell, O’Rourke and Renfrew \cite{campbell2024fractional} provide a "formal proof" of the conjecture without using the i.i.d. assumption. The conjecture is also supported by the fact, proven in Hoskins-Kabluchko \cite{hoskins2021dynamics}, that a similar convergence to the same measure $\mu_t$ occurs under a particular sampling of $\mu_0$, for which $P_N$ has independent coefficients. This result by Hoskins and Kabluchko uses a general result by Kabluchko and Zaporozhets \cite{kabluchko2014asymptotic} on distribution of roots of complex polynomials with independent coefficients: see also \cite{Hughes-Nikeghbali} and \cite{ibragimov2012distribution} for settings where the zeros cluster uniformly around the unit circle. 

In this article, we start by emphasizing a key difference between the one-dimensional and the two-dimensional settings, namely that on the real line, the root sets of $P'_N$ enjoys monotonicity and Lipschitz properties described in the Theorem \ref{monotonicity} in Section \ref{2}, whereas there is no natural total order in the complex plane.

In Section \ref{3}, we define a particular sampling of $\mu_0$, which can be seen as intermediate between one-dimensional and two-dimensional settings. 
More precisely, our sampling of $\mu_0$
depends on two integers $n$ and $m$. 
The roots of $P_N$ are located on $n$ circles of increasing radii $(r_j)_{1 \leq j \leq n}$, centered at the origin. Moreover, on each circle, we choose the arguments of the roots to be $2 i \pi k/m$ for 
$k \in \{0,1,\dots, m-1\}$. The $r_j$'s are chosen as a sampling of the radial marginal of $\mu_0$.  
When one differentiates $m \lfloor n t \rfloor
= \lfloor Nt \rfloor + \mathcal{O}(1)$ times
the polynomial $P_N$, we obtain roots 
located on circles of increasing radii
$(r_{j,t})_{1 \leq j \leq n - \lfloor nt \rfloor} $. Informally, the formula \eqref{Psi}, i.e., $\frac{ \Psi_t^{\langle -1 \rangle}(x)}{x}  =  \frac{ \Psi_0^{\langle -1 \rangle}(x+t)}{x+t} $,
means that 
$$\frac{r_{j,t}}{j}  \simeq \frac{r_{j +  \lfloor nt \rfloor, 0}}{j +  \lfloor nt \rfloor }
=\frac{r_{j +  \lfloor nt \rfloor}}{j +  \lfloor nt \rfloor } $$
for $1 \leq j \leq n - \lfloor nt \rfloor$, 
and then 
$$r_{j-\lfloor nt \rfloor ,t} \simeq r_{j} 
\left( 1 - \frac{nt}{j} \right)$$
for $\lfloor nt \rfloor + 1 \leq j \leq n$. 
For $m$ differentiations, this corresponds to 
$$r_{j-1,1/n} \simeq r_{j} \left(1 - \frac{1}{j} \right).$$

In Section \ref{4}, we prove convergence to an explicitly defined limiting measure $\mu_t$ for the previously defined sampling, as $m$ and $n$ go to infinity, with $m$ growing sufficiently fast with respect to $n$.  More precisely, our assumption is that $m / n \log n$ tends to infinity with $n$. We
describe $\mu_t$ in terms of a generalization of \eqref{Psi} which is available for all probability measures $\mu_0$. We provide sufficient regularity conditions under which \eqref{Psi}, \eqref{equationPsixtabstract} and \eqref{PDErotational} are satisfied. 

In Section \ref{Discussion}, we discuss some historical facts about related conjectures. 
In Section \ref{Examples_prospective}, we provide examples and discuss other problems related to our main results.

\section{Monotonicity and Lipschitz properties on the real line} \label{2}
In this section, we show that in the case where all roots are real, 
taking the derivative of polynomials preserves the natural partial order between sets of roots. 
Such a result is specific to one-dimensional setting and will be useful in our study of the main setting of the article. We also deduce a Lipschitz property for the map giving the roots of the derivative from the roots of the initial polynomial. 

\begin{theorem} \label{monotonicity}
Let $n \geq 2$, $w_1, \dots, w_n > 0$. Then, the function from $$ \Delta_n := \{ (z_1, \dots, z_n)\in  \mathbb{R}^n, z_1 \leq z_2 \leq \dots \leq z_n \} $$ to $\Delta_{n-1}$
such that the image of $(z_1, \dots, z_n)$ 
is the nondecreasing sequence of roots of the polynomial 
$$\sum_{j=1}^n w_j \prod_{1 \leq \ell \leq n, \ell \neq j} (z - z_\ell) \in \mathbb{R}_{n-1}(z),$$
counted with multiplicity, is increasing for the partial order $\preccurlyeq$ given by the following definition: on $\Delta_p$, 
$$(z_1, \dots, z_p) \preccurlyeq (z'_1, \dots, z'_p) $$
if and only if $z_j \leq z'_j$ for $1 \leq j \leq p$.

Moreover, if $w_1 = w_2 = \dots = w_n$, this function is $n/(n-1)$-Lipschitz for the L\'evy metric $L$ between empirical measures, which is defined for two probability measures $\mathbb{P}$ and 
$\mathbb{Q}$ by 
$$L(\mathbb{P}, \mathbb{Q}) = 
\inf \{\varepsilon > 0, \forall x \in \mathbb{R},
F_{\mathbb{P}} (x - \varepsilon ) - \varepsilon \leq F_{\mathbb{Q}}(x) \leq F_{\mathbb{P}} (x + \varepsilon ) + \varepsilon  \},$$
where $F_{\mathbb{P}}$ is the distribution function of $\mathbb{P}$ and $F_{\mathbb{Q}}$ is the distribution function of $\mathbb{Q}$.

\end{theorem}
\begin{proof}
Let us assume 
$$(z_1, \dots, z_n) \preccurlyeq (z'_1, \dots, z'_n)$$
in $\Delta_n$. 
If we take into account multiplicities, for $1 \leq s \leq n-1$, the $s$-th smallest root of the two polynomials 
$$P : z \mapsto \sum_{j=1}^n w_j \prod_{1 \leq \ell \leq n, \ell \neq j} (z - z_\ell)$$ 
and 
$$Q : z \mapsto \sum_{j=1}^n w_j \prod_{1 \leq \ell \leq n, \ell \neq j} (z - z'_\ell)$$ 
lie in the intervals $[z_s, z_{s+1}]$ and $[z'_s, z'_{s+1}]$, respectively. If $z_s = z_{s+1}$ or 
$z'_s = z'_{s+1}$, we have $z_s \leq z_{s+1} \leq z'_s \leq z'_{s+1}$. Hence, the $s$-th root of $P$ is at most 
the $s$-th root of $Q$. We now assume that $z_s < z_{s+1}$
and $z'_s < z'_{s+1}$, and we denote by $\mu$ the $s$-th root of $P$. If $\mu \leq z'_s$, $\mu$ is at most the $s$-th root of $Q$.
Otherwise, $\mu$ is the unique root in $(z'_s, z_{s+1})$
of the rational function 
$$z \mapsto \sum_{j=1}^n \frac{w_j}{z - z_j}.$$
Since $z'_s < \mu < z_{s+1}$, $\mu$ is strictly at the right 
of all points in the interval $[z_r, z'_r]$ for $r \leq s$ and strictly 
at the left for $r \geq s+1$. We deduce that in both cases,
$1/(\mu - z)$ is well-defined and increasing in $z \in [z_r, z'_r]$. Hence, 
$$ \sum_{j = 1}^n \frac{w_j}{ \mu - z'_j} \geq
\sum_{j = 1}^n \frac{w_j}{ \mu - z_j}  = 0.$$
Now, the $s$-th root $\nu$ of $Q$ is the unique root
in $(z'_s , z'_{s+1})$ of the rational function
$$z \mapsto \sum_{j=1}^n \frac{w_j}{z - z'_j}.$$
Since this rational function is decreasing on $(z'_s , z'_{s+1})$ and is nonnegative at $\mu$, which is in this interval, 
we deduce that $\nu \geq \mu$. 
We have proven that the function in the proposition is nondecreasing. 
It is strictly increasing because a simple observation of the two leading coefficients shows that for 
 $$(z_1, \dots, z_n) \preccurlyeq (z'_1, \dots, z'_n)$$ 
 and 
 $$(z_1, \dots, z_n) \neq (z'_1, \dots, z'_n),$$
the sum of the roots of $Q$ is strictly larger than the sum of the roots of $P$. 

For the Lipschitz property in the case 
$w_1 = \dots = w_n$, let us assume that 
for $(z_1, \dots, z_n)$ and 
$(z'_1, \dots, z'_n)$ in $\Delta_n$, the corresponding empirical measures are at distance strictly smaller than $\varepsilon \in (0,1)$. 
In this case, for $1 \leq s \leq n$, applying the definition
of the L\'evy distance to $x < z_s - \varepsilon$,  and letting $x \rightarrow z_s - \varepsilon$, we deduce that the number of points 
among $z'_1, \dots, z'_n$ which are strictly smaller than 
$z_s - \varepsilon$ is at most $n ( (s-1)/n + \varepsilon)$, and then at most $s - 1 + \lfloor n \varepsilon \rfloor $. Hence, 
$z'_{s + \lfloor n \varepsilon \rfloor}  \geq z_s - \varepsilon$ as soon 
as $s + \lfloor n \varepsilon \rfloor \leq n$. 
We deduce that
$$(z'_1 - A, \dots, z'_{ \lfloor n \varepsilon \rfloor} - A, z_1 - \varepsilon, \dots, z_{n -\lfloor n \varepsilon \rfloor} - \varepsilon)
\preccurlyeq (z'_1, \dots, z'_n) $$
where $A > 0$ is sufficiently large, in order to have the left-hand side in $\Delta_n$. 
We then have the same inequality 
between the roots of the polynomials of degree $n-1$
constructed in the statement of the proposition. 
Let us compare the polynomials of degree $n-1$
constructed from 
$$(z'_1 - A, \dots, z'_{ \lfloor n \varepsilon \rfloor} - A, z_1 - \varepsilon, \dots, z_{n -\lfloor n \varepsilon \rfloor} - \varepsilon)$$ and from 
$$ (z_1 - \varepsilon, \dots, z_n - \varepsilon).$$
For $1 \leq r \leq n - 1- \lfloor n \varepsilon \rfloor$, 
the $(r + \lfloor n \varepsilon \rfloor)$-th 
root of the first polynomial, and the $r$-th root of the second polynomial are both in the interval $[z_r - \varepsilon, z_{r+1} - \varepsilon]$, when we take into account multiplicities. If $z_r < z_{r+1}$, these roots are the roots of rational functions, which are decreasing in $(z_r- \varepsilon, z_{r+1}- \varepsilon)$, 
the rational function constructed from 
$(z'_1 - A, \dots, z'_{ \lfloor n \varepsilon \rfloor} - A, z_1 - \varepsilon, \dots, z_{n -\lfloor n \varepsilon \rfloor} - \varepsilon)$ being larger than the rational function constructed from $(z_1 - \varepsilon, \dots, z_n - \varepsilon)$ at each point on $(z_r- \varepsilon, z_{r+1}- \varepsilon)$, because one goes from the first rational function to the second by replacing positive terms by negative terms, keeping the other terms unchanged: notice that here, we use the fact that $w_1 = w_2 = \dots = w_n$.  The $(r + \lfloor n \varepsilon \rfloor)$-th zero of the polynomial constructed 
from $(z'_1 - A, \dots, z'_{ \lfloor n \varepsilon \rfloor} - A, z_1 - \varepsilon, \dots, z_{n -\lfloor n \varepsilon \rfloor} - \varepsilon)$ is then at least equal to the 
$r$-th zero of the polynomial constructed from 
$(z_1- \varepsilon, \dots, z_n - \varepsilon)$. Hence, the 
$(r + \lfloor n \varepsilon \rfloor)$-th zero of the polynomial constructed from $(z'_1, \dots, z'_n)$ is at least the $r$-th zero $\mu_r$ of the 
polynomial constructed from $(z_1, \dots, z_n)$, minus $\varepsilon$. 
If $F$ and $G$ are the distribution functions of the empirical distribution of the roots of the polynomials of degree $n-1$ constructed from 
$(z_1, \dots, z_n)$ and $(z'_1, \dots, z'_n)$, we deduce that for $1 \leq r \leq n - 1 - \lfloor n \varepsilon \rfloor$,  
$x < \mu_r - \varepsilon$, 
and $x \geq  \mu_{r-1} - \varepsilon$ when $r \geq 2$, 
$$G(x) \leq \frac{1}{n-1} (r - 1 + \lfloor n \varepsilon \rfloor) \leq F(\mu_{r-1})   + \frac{n}{n-1} \varepsilon  \leq F(x + \varepsilon) +  \frac{n}{n-1} \varepsilon 
$$
when $r \geq 2$, and 
$$G(x) \leq \frac{1}{n-1} (\lfloor n \varepsilon \rfloor) \leq   \frac{n}{n-1} \varepsilon  \leq F(x + \varepsilon) +  \frac{n}{n-1} \varepsilon 
$$
when $r = 1$. 
We then get 
$$G(x) \leq  F(x + \varepsilon) +  \frac{n}{n-1} \varepsilon $$
for all $x < \mu_{n- 1 - \lfloor n \varepsilon \rfloor} - \varepsilon$, if $n- 1 - \lfloor n \varepsilon \rfloor \geq 1$. 
If $n - 1 -\lfloor n \varepsilon \rfloor \geq 1 $
and $x \geq \mu_{n- 1 - \lfloor n \varepsilon \rfloor}
- \varepsilon$, 
we get 
$$  F(x + \varepsilon) +  \frac{n}{n-1}  \varepsilon
\geq \frac{n - 1 -\lfloor n \varepsilon \rfloor}{n-1}
+ \frac{n}{n-1} \varepsilon \geq 1 \geq G(x),$$
and if $n - 1 - \lfloor n \varepsilon \rfloor = 0$, 
we have $n \varepsilon \geq n-1$ and then for all $x \in \mathbb{R}$, 
$$ F(x + \varepsilon) +  \frac{n}{n-1} \varepsilon
\geq 1 \geq G(x).$$
Hence, in any case, we have proven 
$$G(x) \leq F(x + \varepsilon) + \frac{n}{n-1}  \varepsilon.$$
Now, if we change all the points to their opposite and reverse their order, the distribution functions 
are changed via the map $F \mapsto \widetilde{F}$ where 
$$\widetilde{F} (x) = 1 - F((-x)-),$$
where $F((-x)-)$ denotes the left-limit of $F$ at $-x$,  
and the L\'evy distance between empirical measures does not change: indeed, increasing $\varepsilon$ by an arbitrarily small positive quantity in the definition of the distance absorbs possible errors due to the introduction of left limits of distribution functions. 
Applying the reasoning above after doing this transformation gives, for all $x \in \mathbb{R}$, 
with obvious notation, 
$$1 - G(x-) = \widetilde{G} (-x) \leq \widetilde{F} (-x + \varepsilon) + \frac{n}{n-1} \varepsilon
= 1 - F((x - \varepsilon)-) + \frac{n}{n-1} \varepsilon$$
and then
$$G(x) \geq G(x-) \geq F((x-\varepsilon)-) - \frac{n}{n-1} \varepsilon
\geq F(x-\varepsilon- \eta) - \frac{n}{n-1}
(\varepsilon + \eta),$$
for arbitrarily small $\eta > 0$. 
This inequality, combined with the upper bound above,
implies, after letting $\eta \rightarrow 0$, the Lipschitz property stated in the theorem. 

\end{proof}
The following result shows that, keeping the notation of the introduction of this article, 
the limiting measure $\mu_t$ cannot depend on the sampling of the measure $\mu_0$ in the case where all roots of polynomials are real. 
The situation is different for complex roots: if the uniform distribution on the unit circle is sampled by taking the $n$-th roots of unity, all roots of iterated derivatives are equal to zero, whereas 
sampling according to the setting given by Theorem \ref{main} below gives a different dynamics.

\begin{proposition}
Let $t \in (0,1)$. 
For $n \geq 1$, let $P_n$ and $Q_n$ be degree $n$ polynomials: one assumes that the empirical distribution of the roots of $P_n$ and $Q_n$ both converge to a limiting measure $\mu_0$ when $n \rightarrow \infty$, and 
that the empirical measure of the roots of the $\lfloor t n \rfloor$-th derivative of $P_n$ converges to a limiting measure $\mu_t$. Then, the empirical measure of the roots of the $\lfloor t n \rfloor$-th derivative of $Q_n$ also converges to $\mu_t$.
\end{proposition}
\begin{proof}
One applies $\lfloor t n \rfloor$ times the previous proposition, with all coefficients $w_j$ equal to $1$. 
The L\'evy distance between the empirical measures of the roots of the  $\lfloor t n \rfloor$-th derivatives of $P_n$ and $Q_n$ is at most $n/(n - \lfloor t n \rfloor)$ times the L\'evy distance between the empirical measures of the roots of $P_n$ and $Q_n$, and 
then tends to zero since these empirical measures converge to the same limit $\mu_0$. 
Since the distance to $\mu_t$ of the empirical measure
 of the roots of the  $\lfloor t n \rfloor$-th derivatives of $P_n$ converges to zero, it is then the same for the distance to $\mu_t$ of the empirical measure
 of the roots of the  $\lfloor t n \rfloor$-th derivatives of $Q_n$. 
\end{proof}

\section{Sampling  a  rotationally invariant probability measure} \label{3}
In this section, we define the sampling which is considered in our main Theorem \ref{main} below.
We start with a
rotationally invariant probability measure $\mu_0$, 
which can be written, in polar coordinates, as a tensor product $\nu_0 \otimes unif$, where $\nu_0$ is a probability measure on $\mathbb{R}_+$, and 
$unif$ is the uniform distribution on $[0, 2 \pi)$.
We have for all $s > 0$, 
$$\mu_0 (\mathbb{D}_s) = \nu_0 ([0,s))$$
where $\mathbb{D}_s$ is the open unit disc of center $0$ and radius $s$. 

For couples of positive integers $(n,m)$, we consider, for $N = mn$, a $N$-sample 
of $\mu_0$ defined by 
 $$
\mu^{(n,m)}:= \frac{1}{nm}  \sum_{j=1}^n 
\sum_{k = 0}^{m-1} \delta_{r^{(n)}_j e^{2 i \pi k/m}}
$$
where $(r^{(n)}_j)_{1 \leq j \leq n}$ is a nondecreasing sequence of positive radii sampling the distribution $\nu_0$:
$$\frac{1}{n} \sum_{j=1}^n \delta_{r^{(n)}_j} 
\longrightarrow \nu_0$$
when $n \rightarrow \infty$. 
The sampling has been chosen in such a way that 
the points lie on $n$ circles, each circle having $m$ equidistributed points, in order to approximate the rotational invariance. The points also lie on 
$m$ half-lines, corresponding to arguments 
multiples of $2 \pi/m$. 

 An illustration (Figure \ref{m=15,n=5withP,P'',P(15)}) with $m=15$ and $n=5$, computed with Maple, shows the root set of $P$, $P''$, $P^{(15)}$ and $P^{(45)}$.

\begin{figure}[ht!]
  \includegraphics[width=4cm]{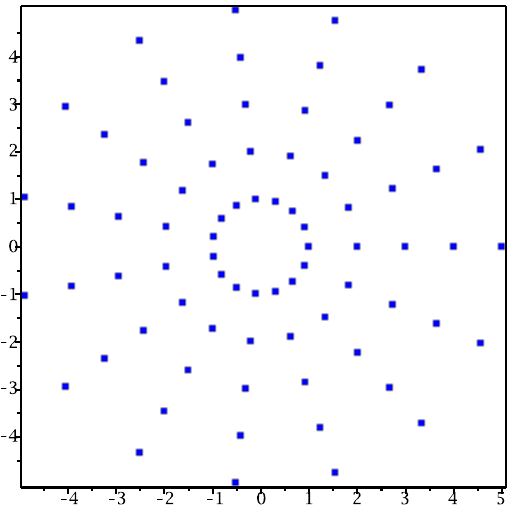}
   \includegraphics[width=4cm]{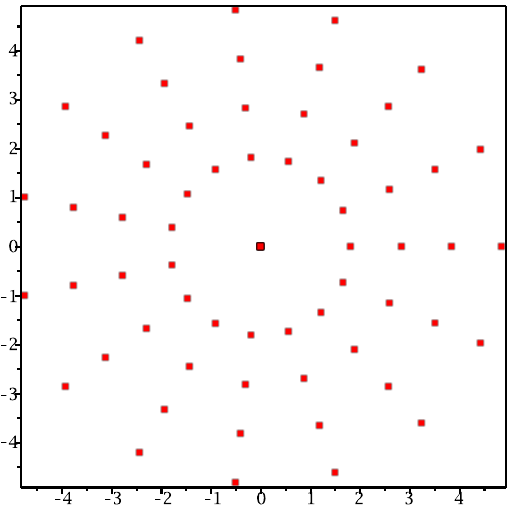}
   \includegraphics[width=4cm]{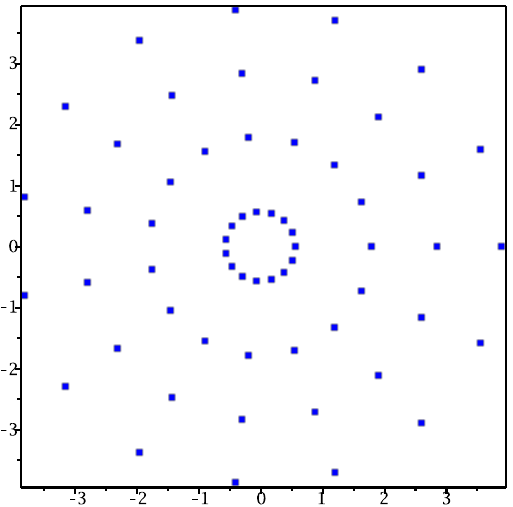}
   \includegraphics[width=4cm]{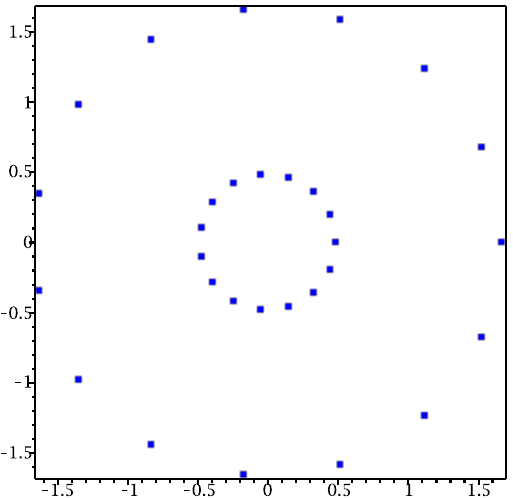}
   \caption{Root sets of $P$, $P^{(2)}$, $P^{(15)}$, and $P^{(45)}$.  }
   \label{m=15,n=5withP,P'',P(15)}
\end{figure}

In this setting, we get the polynomial 
$P_{n,m}$ such that: 
\begin{equation} P_{n,m} (z) = \prod_{j=1}^n 
\prod_{k = 0}^{m-1} (z - r^{(n)}_j e^{2i \pi k/m})
=   \prod_{j=1}^n (z^m - (r^{(n)}_j)^m).
\label{Pnm}
\end{equation}

\begin{lemma} \label{lemmaConstruction}
 Let $n, m \geq 1$, $q \geq 0$ be integers, $Q$ a polynomial of degree $n$, all roots being real and positive, $P$ the polynomial such that $P(z)=z^q Q(z^m)$. Then, the  derivative of $P$ has the following form,  with a polynomial $S$: 
 $$ P'(z) = q\, z^{q-1} \,Q(z^m) + m\, z^{q+m-1} \,Q'(z^m) \, =\,z^{q-1} \,S(z^m)\,  \quad  \text{if}  \quad  q \geq 1,  \quad \text{with}  \quad  \deg(S)=n $$ 
  $$ P'(z) =m  z^{m-1} Q'(z^m) \quad  \text{if}  \quad q=0 ,  \quad  \text{with}  \quad  \deg(Q')=n-1. $$
  Moreover all roots of the polynomials $S$ and $Q'$ are real and positive. The roots of $S$ interlace between $0$ and the roots of $Q$, and 
  the roots of $Q'$ interlace between the roots of $Q$. 
\end{lemma}
\begin{proof}
The first part is immediate, with 
$$S(x) = q Q(x) + m x Q'(x).$$
The second part is a standard result for $Q'$. 
For $S$, in the case where $Q$ has $n$ simple roots, it is a consequence of the fact that
$\frac{m x Q'(x) + q Q(x)}{x Q(x)}$ has $n+1$ simple poles at 
$0$ and the roots of $Q$, hence $n$ roots interlacing between the poles. By continuity, the lemma remains true when $Q$ has multiple roots. 
\end{proof}

As a consequence of the lemma, after $m$ differentiations of $P_{n,m}$, we get a
polynomial of the same shape, with $n$ replaced by $n-1$, and different values of the radii, given 
by a nondecreasing sequence $(R_j^{(n)})_{1 \leq j \leq n-1}$, such that $R_j^{(n)} \leq r_{j+1}^{(n)}$
for $1 \leq j \leq n-1$. 
In order to prove the variant of the conjecture we study, we will compare $R_j^{(n)}$ to 
$(1 - 1/j) r_{j+1}^{(n)}$ in a quantitative way. 
 We are then led to analyze what happens during a sequence of $m$ differentiations between the 
 $(\ell m)$-th and the $((\ell+1) m)$-th 
 derivatives of $P_{n,m}$, for $0 \leq \ell \leq n-1$. 

 For this purpose, starting with a polynomial $Q$, all its roots being real and positive, we will have to consider, as detailed in the next section, a sequence $(Q_k)_{1 \leq k \leq m} 
$ of polynomials of degree one less that the degree of $Q$, such that $Q_1 = m Q'$ and 
$$ Q_{k+1}(z):= m z  \,Q_k'(z) + (m-k) \, Q_k(z)
$$ 
for $1 \leq k \leq m-1$. 
\vspace{1cm}

\section{Statement and proof of the main theorem} \label{4} 
The main result of the article is stated as follows. 
\begin{theorem} \label{main}
Let $\nu_0$ be a probability measure on $\mathbb{R}_+$ with compact support. 
For $n \geq 1$, let $(r_j^{(n)})_{1 \leq j \leq n}$ an increasing sequence in $\mathbb{R}_+$ such that
$$\frac{1}{n} \sum_{j=1}^n \delta_{r_j^{(n)}} 
\underset{n \rightarrow \infty}{\longrightarrow}\nu_0.$$
Then, for any sequence $(m_n)_{n \geq 1}$
such that $$\frac{m_n}{n \log n} \underset{n \rightarrow \infty}{\longrightarrow} \infty,$$ and for $t \in (0,1)$, the empirical measure of the roots of the $\lfloor n m_n t \rfloor$-th derivative of the polynomial 
$P_{n,m_n}$ defined by \eqref{Pnm} tends 
to the measure $\mu_t = \nu_t \otimes unif$ when $n \rightarrow \infty$, 
where $\nu_t$ is the distribution 
of 
 \begin{align}
 \left(1 - \frac{t}{V_t} \right) q_{\nu_0} (V_t), \label{nuV}
 \end{align}
 $V_t$ being a uniform random variable on $[t,1]$,
$q_{\nu}$ denoting the quantile function of a given finite measure $\nu$: 
\begin{align}
    q_{\nu_0}(\alpha) = \inf \{ y \geq 0, \nu_0 ([0,y]) \geq \alpha \}
\end{align}
for $\alpha \in [0,1]$. 
Moreover, the quantiles of the measure $(1-t)\nu_t$, which has total mass $1-t$, satisfy the 
equation: 
\begin{equation}
    \label{eqquantiles}
 q_{(1-t) \nu_t} (x)  = 
\frac{ x q_{\nu_0}(x+t)}{x+t}
\end{equation}
for $0 \leq x \leq 1-t$. 

In the case where the distribution function 
of $\nu_0$ is a continuous and strictly increasing bijection from $[0,1]$ to  $[0,A]$ for some $A > 0$, the distribution function of $(1-t)\nu_t$ is a continuous and strictly increasing bijection from $[0,1-t]$ to $[0,A(1-t)]$. In this case, these distributions functions $\Psi_0$ and $\Psi_t$ have $q_{\nu_0}$ and $q_{(1-t) \nu_t}$ as reverse bijections, 
respectively from $[0,1]$ to $[0,A]$ and from $[0,1-t]$ to $[0, A(1-t)]$. Hence, \eqref{eqquantiles} implies \eqref{Psi} in this case.

In the case where 
$\nu_0$ is absolutely continuous with respect to the Lebesgue measure on $[0,A]$, 
with a continuous, strictly positive density 
on $(0,A)$, the measure 
$(1-t) \nu_t$ is, for all $t \in [0,1)$.  supported on the interval 
$[0, A(1-t)]$, has a continuous distribution 
function $\Psi_t$, with a strictly positive 
derivative $x \mapsto \psi(t,x)$
on $(0, A(1-t))$, which is a density of $(1-t) \nu_t$ with respect to the Lebesgue measure.  Moreover, 
$$(t, x) \mapsto \Psi_t(x)$$ is a continuously differentiable function of two variables on the set 
$$\{(t,x) \in [0,1) \times \mathbb{R}, 
x \in (0, A(1-t)) \},$$
and the following partial differential equation is satisfied on the same set: 
\begin{equation}
\label{equationPsixt}
\frac{\partial \Psi_t (x)}{\partial t} 
= x \frac{\frac{\partial \Psi_t (x)}{\partial x}  } {\Psi_t(x)} - 1. 
\end{equation}
Moreover, if one makes the extra assumption 
that the density of $\nu_0$ is continuously differentiable on $(0,1)$, then the 
density $\psi$ is a continuously 
differentiable function in two variables on 
the set 
$$\{(t,x) \in [0,1) \times \mathbb{R}, 
x \in (0, A(1-t)) \},$$
satisfying the partial differential equation: 
\begin{equation}
\frac{\partial \psi}{\partial t}(x,t) 
= \frac{\partial}{\partial x} \left( 
\frac{\psi(x,t)}{\frac{1}{x} \int_0^x \psi(y,t) dy} 
\right),
\label{PDErotationalmaintheorem}
\end{equation}

\end{theorem}
\begin{comment}
\begin{remark}
    If $\nu_0$ is not absolutely continuous (e.g., has atoms) or $\psi_0$ is strictly positive, $\psi(x,t)$ may not be differentiable everywhere. The equation \eqref{PDErotational} still holds in the distributional sense. 
\end{remark}
\end{comment}
\subsection{Lemmas}
The proof of Theorem \ref{main} is obtained by applying a series of lemmas, stated below. We keep the notation of Lemma
\ref{lemmaConstruction}. 
\begin{lemma} \label{monotonicitym}
For fixed integers $n, m \geq 1$, $q \geq 1$, the 
map from $\mathbb{R}^n$ to 
$\mathbb{R}^n$ giving the nondecreasing sequence of roots of $S$ in terms of the nondecreasing sequence of roots of $Q$ is increasing for the partial order 
defined in Theorem \ref{monotonicity}. 
For fixed $n, m \geq 1$ and for $q =0$, the map 
from $\mathbb{R}^n$ to 
$\mathbb{R}^{n-1}$ giving the roots of $Q'$ is increasing for the same partial order. 
\end{lemma}
\begin{proof}
For $q \geq 1$, let $z_2 \leq z_3\leq \dots \leq z_{n+1}$
be the roots of $Q$, counted with multiplicity, 
and let $z_{1} := 0$.
We have 
$$S(z) = q Q(z) + m z Q'(z)
= q \prod_{1 \leq j \leq n+1, j \neq 1} 
(z - z_j) 
+ m \sum_{2 \leq \ell \leq n+1} 
\prod_{1 \leq j \leq n+1, j \neq \ell} 
(z - z_j).
$$
Here, we have written the factor $z$ of the second term as $z - z_1$. 
We deduce the lemma from Theorem \ref{monotonicity}
applied to polynomials of degree $n+1$, 
$w_1 = q$ and $w_j = m$ for $2 \leq j \leq n+1$. 
Similarly, the case $q = 0$ is solved by applying 
Theorem \ref{monotonicity}
to polynomials of degree $n$ and weights all equal to $1$. 
\end{proof}

\begin{lemma} \label{upperbound}
For an increasing sequence $(r_j)_{1 \leq j \leq n}$ of positive reals, we assume that for $1 \leq j \leq n$, 
the $j$-th smallest root of $Q$, counted with multiplicity, is at most $r_j$. 
Then, for $1 \leq q \leq m-1$, $1 \leq j \leq n$, 
the $j$-th smallest root of $S$ is 
at most $r_j$ and also at most $ j r_j/((j+1)(1 - \alpha))$, 
where $\alpha$ is the maximum of $r_j/r_{j+1}$
for $1 \leq j \leq n-1$. For $q = 0$, 
$2 \leq j \leq n$, 
the $(j-1)$-th smallest root of $Q'$ is 
at most $r_j$ and at most $ j r_{j}/((j+1)(1 - \alpha))$.
\end{lemma}

\begin{proof}
By Lemma \ref{monotonicitym}, 
we can assume that the $j$-th smallest root of $Q$
is exactly $r_j$. 
By the intermediate value theorem, for $q \geq 1$, 
the roots of $S$ are given by an increasing
sequence $(z_k)_{1 \leq k \leq n}$
such that 
$$r_0 := 0 < z_1 < r_1 < z_2 < r_2 < \dots
< z_n < r_n$$
and 
\begin{equation} \frac{q}{z_k} + \sum_{j=1}^n \frac{m}{z_k - r_j} 
= 0. \label{rootsS}
\end{equation}
For $q = 0$, the roots of 
$Q'$ are $(z_k)_{2 \leq k \leq n}$
where 
$$ r_1 < z_2 < r_2 < \dots
< z_n < r_n$$
and \eqref{rootsS} is satisfied.  
This gives the upper bound $r_j$. 

If for $1 \leq p \leq n$, we discard the terms $j > p$, 
the left-hand side of \eqref{rootsS} above  
increases when $k \leq p$, because $z_k - r_j < 0 $ for $j > p$. Hence, 
$$\frac{q}{z_k} + \sum_{j = 1}^{p} \frac{m}{z_k - r_j}  \geq  0.$$
We then have $z_k \leq z'_k$, where 
$ r_{k-1} < z'_k < r_k$ and 
$$\frac{q}{z'_k} + \sum_{j = 1}^{p} \frac{m}{z'_k   - r_j}  =  0$$
 i.e. 
 $$q \prod_{j=1}^{p} (z'_k - r_j)
 + m z'_k \sum_{j=1}^{p} \prod_{1 \leq \ell  \leq p, \ell \neq j}  (z'_k  - r_\ell) = 0.$$
This equation in $z'_k$ has degree $p$:
for $q \geq 1$, this degree is the number of solutions $z'_k$ we are considering, for $q = 0$, we look for 
$p -1$ solutions $z'_k$ since we need $2 \leq k \leq p$: in this case, the equation has exactly one more solution, namely zero.  
In all cases, the sum of $z'_k$ for $k \leq p$ is the sum of all solutions of the equation, and then looking at the two highest degree coefficients, we get: 
$$\sum_{k \leq p} z'_k 
= \frac{1}{q + mp} \left( q \sum_{j=1}^p r_j 
+ m \sum_{j=1}^p \sum_{1 \leq \ell \leq p, \ell \neq j} r_{\ell}\right) = \frac{q + m(p-1)}{q + mp} 
\sum_{j \leq p} r_j.$$
We deduce 
$$ \sum_{j \leq p} z_j \leq \left(1 - \frac{m}{q + mp} \right) \sum_{j \leq p} r_j.$$
Hence, taking only one term in the left-hand side 
and using the fact that $r_j \leq \alpha r_{j+1}$ 
for $1 \leq j \leq n-1$
by definition of $\alpha$, 
$$z_p \leq 
\left(1 - \frac{m}{q + mp} \right) 
\frac{r_p}{1 - \alpha} 
\leq \left(1 - \frac{m}{m + mp} \right) 
\frac{r_p}{1 - \alpha},
$$
which proves the lemma. 
\end{proof}
\begin{lemma} \label{lowerbound}
For an increasing sequence $(r_j)_{1 \leq j \leq n}$ of positive reals, we assume that for $1 \leq j \leq n$, 
the $j$-th smallest root of $Q$, counted with multiplicity, is at least $r_j$. 
Then, for $1 \leq q \leq m-1$, $1 \leq j \leq n$, 
the $j$-th smallest root of $S$ is 
at least $ \beta r_j$, where 
$$\beta =  1 - \frac{1}{ \max \left(1,  j-1 - \frac{2 + \log_- (\log (\alpha^{-1}) )}{\log (\alpha^{-1})}   \right)}, $$
$\log_-(x):=\max(0,-\log x)$, $\alpha$ being the maximum of $r_j/r_{j+1}$
for $1 \leq j \leq n-1$. For $q = 0$, 
$2 \leq j \leq n$, 
the $(j-1)$-th smallest root of $Q'$ is 
at least $ \beta r_j$. 
\end{lemma}

\begin{proof}
By Lemma \ref{monotonicitym}, we can again assume that the $j$-th smallest root of $Q$ is exactly $r_j$.
We can also assume $2 \leq j \leq n$, since $\beta = 0$ for $j = 1$. 
Keeping the notation of the proof of Lemma \ref{upperbound}, we deduce, from \eqref{rootsS},
\begin{align} 
\frac{q + m (j-1)}{ r_{j}} +
\frac{m}{z_j - r_j} + \sum_{j +1 \leq \ell \leq n}
\frac{m}{r_j - r_\ell}< \frac{q}{z_j}+ \sum_{\ell=1}^{j-1}\frac{m}{z_j-r_{\ell}}+\frac{m}{z_j-r_j}+ \sum_{j +1 \leq \ell \leq n}
\frac{m}{z_j - r_\ell} = 0
\end{align}
for $2 \leq j \leq n$. 
 Since $r_\ell \geq r_j \alpha^{- (\ell-j)} $
 for $\ell \geq j$,
$$ q + m (j-1)  +
\frac{m}{(z_j/r_j) -1} - \sum_{\ell = 1}^{\infty}
\frac{m}{ \alpha^{- \ell}- 1} \leq 0.$$
The last sum is at most 
$$\frac{m}{\alpha^{-1} - 1}
+ \int_{1}^{\infty} \frac{m}{ e^{ x \log (1/\alpha)} - 1} dx
= \frac{m}{\alpha^{-1} - 1}
+ \int_{ \log (1/\alpha)}^{\infty} \frac{m \, dy}{\log (1/\alpha) (e^{y} - 1)} 
$$ $$\leq \frac{m}{\alpha^{-1} - 1} 
+ \frac{m} {\log (1/\alpha)} \left(\int_{\min(\log (1/\alpha), 1)}^1 \frac{dy}{y} + \int_1^{\infty} \frac{dy}{e^y/2} \right) 
$$ $$ \leq \frac{ m \left(2 + 
\log_- ( \log (\alpha^{-1})) \right)}{\log (\alpha^{-1})}.$$
We deduce 
$$ q + m (j-1) -\frac{ m \left(2 + 
\log_- (  \log (\alpha^{-1})) \right)}{\log (\alpha^{-1})} \leq \frac{m}{1 - (z_j/r_j)},  $$
$$j-1 - \frac{  \left(2 + 
\log_- (  \log (\alpha^{-1})) \right)}{\log (\alpha^{-1})} \leq \frac{1}{1 - (z_j/r_j)},
$$
and then, in the case where $\beta > 0$
(the case $\beta = 0$ is trivial), 
$$\frac{1}{1-\beta} \leq \frac{1}{1 - (z_j/r_j)},$$
which proves the lemma. 
\end{proof}
We now state a lemma estimating the effect of $m$ differentiations. We keep the notation of Section \ref{3}. 
\begin{lemma} \label{boundsmderivations}
 The $m$-th derivative of $P_{n,m}$
can be written as 
$$P_{n,m}^{(m)} 
= \frac{(nm)!}{((n-1)m)!} 
\prod_{j=1}^{n-1} (z^m - (R_j^{(n)})^m)$$
for a nondecreasing sequence 
$(R_j^{(n)})_{1 \leq j \leq n-1}$ of positive numbers. Let $(r_j)_{1 \leq j \leq n}$ be an increasing sequence of positive real numbers, and 
let $\alpha$ be the maximum of $r_j/r_{j+1}$ for 
$1 \leq j \leq n-1$. 
If $r_j^{(n)} \leq r_j$ for $1 \leq j \leq n$,
then 
for $1 \leq j \leq n-1$, 
$$R_j^{(n)} \leq \min \left(1, \frac{j+1}{(j+2)(1- \alpha^m)} \right) r_{j+1}.$$
If $r_j^{(n)} \geq r_j$ for $1 \leq j \leq n$, 
then 
for $1 \leq j \leq n-1$, 
$$ R_j^{(n)} \geq \left( 1 - \frac{1}{ \max \left(1,  j - 1 - \frac{2 + \log_- (m \log (\alpha^{-1}) )}{m \log (\alpha^{-1})}   \right)} \right) r_{j+1}.$$
\end{lemma}
\begin{proof}

We have 
$$P_{n,m}(z) = Q_0(z^m)$$
where 
$$Q_0(z) = \prod_{j=1}^n (z - (r_j^{(n)})^m).$$
Iterating $m$ times the computation in Lemma \ref{lemmaConstruction}, we find that for $1 \leq k \leq m$, 
$$P_{n,m}^{(k)} = z^{m-k} Q_k (z^m),$$
where the polynomials $(Q_k)_{1 \leq k \leq n}$ have degree $n-1$ and satisfy $Q_1 = m Q_0'$, and 
$$Q_{k+1} (z) = (m-k) Q_k(z) + m z Q_k'(z) $$
for $1 \leq k \leq m-1$. 
Iterating $m$ times Lemma \ref{upperbound}, with values of $q$ successively
equal to $0, m-1, m-2, \dots, 2, 1$, we
deduce the general form of the factorization of the polynomial $P_{n,m}^{(m)}$. 

Let us now assume $r_j^{(n)} \leq r_j$ for $1 \leq j \leq n$: in this case, the $j$-th smallest root 
of $Q_0$ is bounded by $r_j^m$ for $1 \leq j \leq n$.  
Iterating $m$ times Lemma \ref{upperbound},
we deduce, for $1 \leq j \leq n-1$, 
successive upper bounds 
on the $j$-th smallest root 
of $Q_1, Q_2, \dots, Q_m$, from the 
fact that the $(j+1)$-th smallest root  
of $Q_0$ is at most $r_{j+1}^m$. 
More precisely, we get by induction, that
for $1 \leq k \leq m$, the $j$-th smallest root 
of $Q_k$ is at most 
$$ \min \left(1, \frac{j+1}{(j+2)(1-\alpha^m)} \right) \, \min \left(1, \frac{j}{(j+1)(1-\alpha^m)}  \right)^{k-1} r_{j+1}^m.$$
Notice that in this induction, we use the fact that the ratio between these upper bounds for consecutive values of $j$ always remains bounded 
by $\alpha^m$, which is true because 
$r_j^m /r_{j+1}^m \leq \alpha^m$ by assumption, 
and $j/(j+1)$, $(j+1)/(j+2)$ are increasing in $j$. Since the $j$-th smallest root of $Q_m$ is
$(R_j^{(n)})^m$, we have 
$$ (R_j^{(n)})^m
\leq \min \left(1, \frac{j+1}{(j+2)(1-\alpha^m)} \right) \, \min \left(1, \frac{j}{(j+1)(1-\alpha^m)}  \right)^{m-1} r_{j+1}^m,$$
$$ (R_j^{(n)})^m
\leq \min \left(1, \frac{j+1}{(j+2)(1-\alpha^m)} \right)^m r_{j+1}^m,
$$
which proves the upper bound of Lemma \ref{boundsmderivations}. The lower bound is 
exactly proven in the same way, 
using Lemma \ref{lowerbound} instead of Lemma \ref{upperbound}. 

\end{proof}
\begin{lemma} \label{46}
Under the assumption of Theorem \ref{main}, 
we have for fixed $t \in (0,1)$ and $n$ sufficiently large, $$\lfloor n m_n t \rfloor = \ell m_n - q$$
where $1 \leq \ell \leq n-1$,  $nt - 1 \leq \ell \leq nt + 1$, $0 \leq q \leq m_n-1$.
Moreover, the roots of the $\lfloor n m_n t \rfloor$-th derivative of $P_{n,m_n}$, repeated 
according to their multiplicity, are
$0$ with multiplicity $q$, and 
$s_j e^{2i \pi k/m_n}$
for $1 \leq j \leq n-\ell$, $0 \leq k \leq m_n-1$, where for $4 \leq j \leq n-\ell$, 
$$e^{-\eta_j} \frac{j-1}{j-1 + nt} 
r_{j+\ell}^{(n)} \leq s_j \leq 
e^{\eta_j} \frac{j-1}{j-1 + nt} r^{(n)}_{j+\ell},$$
$(\eta_j)_{j \geq 4}$ being a decreasing sequence 
depending only on the sequence $(m_n)_{n \geq 1}$ and tending to zero when $j \rightarrow \infty$.

\end{lemma}
\begin{proof}
    Iterating Lemma \ref{boundsmderivations}, 
we deduce that for $1 \leq \ell \leq n-1$, 
$$P_{n,m}^{(\ell m)} (z)
= \frac{(nm)!}{((n-\ell)m)!} 
\prod_{j=1}^{n-\ell} (z^m -  (r_j^{(n, m, \ell)} )^m )$$
where for $1 \leq j \leq n- \ell$,
$$r_j^{(n, m,\ell)} 
\leq r_{j+\ell} \prod_{s = 1}^{\ell} 
\min \left(1, \frac{j+s}{(j+s+1) (1 - \alpha^m)} \right)  
$$
as soon as $r_j^{(n)} \leq r_j$ for $1 \leq j \leq n$, $(r_j)_{1 \leq j \leq n}$ being an increasing sequence of positive reals, $\alpha$ being the maximum of $r_j/r_{j+1}$ for $1 \leq j \leq n-1$.
We deduce 
$$r_j^{(n,m,\ell)} \leq \frac{j+1}{(j+ \ell +1 )(1 - \alpha^m)^{\ell}} r_{j+\ell}.$$
For $\gamma > 1$, we can apply this 
result to $r_j = r_j^{(n)} \gamma^j$, 
in which case $\alpha \leq \gamma^{-1}$, and then 
$$r_j^{(n,m,\ell)} \leq \frac{\gamma^{j+ \ell}(j+1)}{(j+ \ell +1 )(1 - \gamma^{-m})^{\ell}}  r^{(n)}_{j+\ell}.$$
Taking, for $n \geq 2$, $\gamma = e^{ 3 m_n^{-1} \log n}$, we deduce, 
since $j+ \ell \leq n$, 
$$r_j^{(n,m_n,\ell)} \leq \frac{e^{3 m_n^{-1} n \log n}(j+1)}{(j+ \ell +1 ) \left(1 - e^{-3 \log n} \right)^{n}}  r^{(n)}_{j+\ell}.$$
 Now, 
$$(1 - e^{-3 \log n})^n \geq 1 - n e^{-3 \log n} 
= 1 - \frac{1}{n^2} \geq 1 - \frac{1}{n+1} \geq 1 - 
\frac{1}{j+\ell+1} = \frac{j+\ell}{j+\ell+1}$$
and then 
$$r_j^{(n, m_n, \ell)} \leq \frac{j+1}{j+\ell} 
e^{3 m_n^{-1} n \log n} r_{j+\ell}^{(n)}.$$
Similarly, if $r_j^{(n)} \geq r_j$
for $1 \leq j \leq n$, we get 
$$r_j^{(n,m,  \ell)} 
\geq r_{j+\ell} \prod_{s = 1}^{\ell} 
\left( 1 - \frac{1}{ \max \left(1,  j +s - 2 - \frac{2 + \log_- (m \log (\alpha^{-1}) )}{m \log (\alpha^{-1})}   \right)} \right)$$
for $1 \leq j \leq n - \ell$. 
Now, for $n \geq 3$, we take $r_j = r_j^{(n)} \gamma^{j-n}$, $\gamma = e^{3 m_n^{-1} \log n}$. 
Since $\alpha \leq \gamma^{-1}$, we get 
$$2 \leq 2 + \log_- (m_n \log (\alpha^{-1}))
\leq 2 + \log_- (m_n \log \gamma)
= 2 + \log_- (3 \log n) 
\leq 2 + \log_- (3 \log 3) = 2,$$
and 
$$m_n \log (\alpha^{-1}) \geq m_n \log \gamma
= 3 \log n \geq 3,$$
which implies 
$$j + s - 2 - \frac{2 + \log_- (m \log (\alpha^{-1}) )}{m \log (\alpha^{-1})}  
\geq j + s -3.$$
 We deduce, for $3 \leq j \leq n- \ell$, 
 $$r_j^{(n, m_n, \ell)}  
 \geq \frac{j-3}{j+\ell - 3} r^{(n)}_{j+ \ell}
 e^{-3 m_n^{-1} n \log n}. 
 $$
Now, for $1 \leq \ell \leq n-1$
and $0 \leq q \leq m-1$,
we can write 
$$P_{n,m}^{(\ell m - q)} (z) 
= \frac{(nm)!}{((n-\ell)m + q)!}
z^q \prod_{j=1}^{n-\ell}
(z^m - (r_j^{(n,m,\ell-q/m)})^m),
$$
where, from the simplest upper bound given by Lemma \ref{upperbound}, 
$$r_j^{(n,m, \ell)} \leq 
r_j^{(n,m,\ell-q/m)} \leq r_{j+1}^{(n,m,\ell-1)},$$
for $r_{j+1}^{(n,m,0)} = r_{j+1}^{(n)}$ in the case $\ell = 1$.
Hence, for $3 \leq j \leq n - \ell$, 
$0 \leq q \leq m_n-1$, 
$$\frac{j-3}{j+\ell-3}e^{-3 m_n^{-1} n \log n} 
r_{j+\ell}^{(n)} \leq r_j^{(n,m_n,\ell-q/m_n)} \leq \frac{j+2}{j+\ell} e^{3 m_n^{-1} n \log n} r_{j+\ell}^{(n)}.$$
For fixed $t \in (0,1)$ and $n$ sufficiently large, 
$$\lfloor n m_n t\rfloor = \ell m_n - q$$
for $1 \leq \ell \leq n-1$ and $0 \leq q \leq m_n-1$. Moreover, 
$$nt - 1 \leq \ell \leq nt + 1.$$
The roots of $P_{n,m_n}^{(\lfloor n m_n t \rfloor)}$ are then zero with multiplicity 
$q$, and 
$s_j e^{2 i \pi k/m_n}$
for $1 \leq j \leq n - \ell = n (1-t) + \mathcal{O}(1)$, $0 \leq k \leq m_n-1$, 
where for $j \geq 3$, 
$$\frac{j-3}{j + nt - 1} e^{-\varepsilon_n}
 r^{(n)}_{j+ \ell} \leq \frac{j-3}{j + nt - 2} e^{-\varepsilon_n}
 r^{(n)}_{j+ \ell}
\leq s_j \leq \frac{j+2}{j+ nt -1} e^{\varepsilon_n} r^{(n)}_{j+ \ell}.$$
Here, by assumption on the sequence $(m_n)_{n \geq 1}$, $\varepsilon_n := 3 m_n^{-1} n \log n$ tends to zero when $n \rightarrow \infty$. Hence, 
for $4 \leq j \leq n-\ell$, 
$$ e^{-\eta_j} \frac{j-1}{j-1 + nt} r_{j+ \ell}^{(n)} \leq  s_j \leq e^{\eta_j} \frac{j-1}{j-1 + nt} r_{j+ \ell}^{(n)} $$
where 
$$\eta_j := \log ((j-1)/(j-3)) + 
\log ((j+2)/(j-1)) 
+ \sup_{p \geq j} \varepsilon_p
$$
decreases to zero when $j$ tends to infinity. 
\end{proof}
\subsection{End of proof of Theorem \ref{main}}
We now complete the proof of Theorem \ref{main}. 
By Lemma \ref{46}, if the roots of $P_{n,m_n}^{(\lfloor n m_n t \rfloor)}$ are moved by changing $s_j$ to 
$(j-1)/(j-1 + nt) r_{j+ \ell}^{(n)}$, 
then for $4 \leq j_0 \leq n- \ell$, 
$1 \leq j_1 \leq n-1$,
a proportion of 
at most $(j_0 + j_1) m_n / (n m_n - \lfloor n m_n t\rfloor)$ of the roots are moved 
by more than $(e^{\eta_{j_0}} - 1) r^{(n)}_{n- j_1}$: here, $j_0$ corresponds to the number of radii $r_j$ for which $\eta_j$ can be
larger than $\eta_{j_0}$, and $j_1$
the number of radii for which $r_{j+\ell}^{(n)}$ can be larger than $r^{(n)}_{n- j_1}$. 
Taking, for $\varepsilon \in (0, 1/10)$, 
$j_0 = j_1 = \lfloor n (1-t) \varepsilon \rfloor$, we get for fixed $\varepsilon, t$ and $n$ large enough, $e^{\eta_{j_0}} - 1 \leq \varepsilon$, and 
$r^{(n)}_{n- j_1} \leq A + 1$
if $\nu_0$ is supported on $[0, A]$, 
because the empirical measure 
of $(r_j^{(n)})_{1 \leq j \leq n}$ tends to $\nu_0$ by assumption, and then 
$o(n)$ points among $(r_j^{(n)})_{1 \leq j \leq n}$ can be larger than $A+1$. 
For fixed $\varepsilon, t$, and for $n$ large enough, 
a proportion at most $3 \varepsilon$
of the roots are moved by more than $(A+1) \varepsilon$. Letting $\varepsilon \rightarrow 0$, we deduce that the L\'evy-Prokhorov distance between the empirical measure of 
the roots of $P_{n,m_n}^{(\lfloor n m_n t \rfloor)}$ and the points obtained from these roots by changing $s_j$ to $(j-1)/(j-1 + nt) r_{j+ \ell}^{(n)}$ tends to zero when $n \rightarrow \infty$. In order to show 
convergence of the empirical measure of the
roots of $P_{n,m_n}^{(\lfloor n m_n t \rfloor)}$, it is then enough 
to show convergence, when $n \rightarrow \infty$, 
of the measure 
$$\frac{1}{q + (n- \ell) m_{n}} \left( q \delta_0 + \sum_{j=1}^{n-\ell} \sum_{k=0}^{m_{n}-1}
\delta_{e^{2 i\pi k/m_n} r_{j+\ell}^{(n)} (j-1)/(j-1 + nt) } \right)
$$
towards $\mu_t$: 
notice that $q + (n-\ell) m_n = n m_n - \lfloor n m_n t\rfloor$ is the number of 
roots of $P_{n, m_n}^{(\lfloor n m_n t \rfloor)}$. 
We can rotate the measure by an 
angle between $0$ and $2 \pi /m_n$, keeping a L\'evy-Prokhorov distance tending to zero, since we move a proportion tending to one of the points 
by $\mathcal{O} ((A+1) /m_n)$. Averaging 
among the possible angles, it is enough to show 
convergence 
$$\frac{1}{q + (n- \ell) m_{n}}
\left(q \delta_0 + m_n \sum_{j=1}^{n-\ell}
\delta_{ r_{j+\ell}^{(n)} (j-1)/(j-1 + nt) }
\right) \otimes unif
\underset{n \rightarrow \infty}{\longrightarrow } \mu_t = \nu_t \otimes unif$$
i.e. 
$$\frac{1}{q + (n- \ell) m_{n}}
\left(q \delta_0 + m_n\sum_{j=1}^{n-\ell}
\delta_{ r_{j+\ell}^{(n)} (j-1)/(j-1 + nt) }
\right) 
\underset{n \rightarrow \infty}{\longrightarrow }  \nu_t.$$
By moving a negligible part of the measure, one deduces that it is enough to prove 
$$\frac{1}{n - \ell} 
\sum_{j=1}^{n-\ell} \delta_{ r_{j+\ell}^{(n)} (j-1)/(j-1 + nt) }\underset{n \rightarrow \infty}{\longrightarrow }  \nu_t.$$
The left-hand side is the distribution of 
$$r^{(n)}_{\ell + 1 + \lfloor (n-\ell) U \rfloor}
\frac{\lfloor (n-\ell) U \rfloor}{\lfloor (n-\ell) U \rfloor + nt}
= q_{\nu^{(n)}} \left( \frac{ \ell + 1 + \lfloor (n-\ell) U \rfloor}{n} \right)
\frac{\lfloor (n-\ell) U \rfloor}{\lfloor (n-\ell) U \rfloor + nt}
$$
where $U$ is uniformly distributed on $[0,1]$, 
and $\nu^{(n)}$ is the empirical distribution 
of $(r^{(n)}_j)_{1 \leq j \leq n}$, i.e. 
$$\nu^{(n)}  = \frac{1}{n} \sum_{j=1}^{n} \delta_{r^{(n)}_j}. $$
Since $nt - 1 \leq \ell \leq nt + 1$, 
we get for $n > 3/t$, and then $3/n < t$, 
$$ q_{\nu^{(n)}} \left(   t + (1-t)U - 3/n   \right)
 \leq q_{\nu^{(n)}} \left( \frac{ \ell + 1 + \lfloor (n-\ell) U \rfloor}{n} \right)
\leq q_{\nu^{(n)}} \left( \min(1, t + (1-t)U + 3/n)   \right). $$
Now, using L\'evy-Prokhorov distance 
and convergence of $(\nu^{(n)})_{n \geq 1}$
towards $\nu_0$, we deduce that 
for fixed $\alpha \in [0,1]$, $\varepsilon > 0$, 
and for $n$ large enough, 
$$ \nu^{(n)} ([0, q_{\nu_0} (\alpha) + \varepsilon]) + \varepsilon \geq \nu_0 ([0, q_{\nu_0} (\alpha)]) \geq \alpha, \;
\nu_0 ([0, q_{\nu^{(n)}} (\alpha) + \varepsilon]) + \varepsilon \geq \nu^{(n)} ([0, q_{\nu^{(n)}} (\alpha)]) \geq \alpha,
$$
and then 
$$ q_{\nu^{(n)}} (\max(0,\alpha - \varepsilon))
\leq q_{\nu_0} (\alpha) + \varepsilon, \; 
q_{\nu^{(n)}}(\alpha) \geq   q_{\nu_0} (\max(0,\alpha - \varepsilon)) - \varepsilon.
$$
Hence, for all $\varepsilon \in (0,t)$, 
$$\underset{n \rightarrow \infty}{\lim \inf} \, 
q_{\nu^{(n)}} \left( t + (1-t)U - 3/n   \right)
\geq \underset{n \rightarrow \infty}{\lim \inf} \, 
q_{\nu^{(n)}} \left( t + (1-t)U - \varepsilon/2  \right)
\geq   
q_{\nu_0}  \left( t + (1-t)U - \varepsilon \right)
- \varepsilon/2,
$$
and then, letting $\varepsilon \rightarrow 0$, 
$$\underset{n \rightarrow \infty}{\lim \inf} \, 
q_{\nu^{(n)}} \left( t + (1-t)U - 3/n   \right)
\geq q_{\nu_0}  \left( ( t + (1-t)U)-  \right)
$$
where $q_{\nu_0}(\alpha-)$ is the limit of $q_{\nu_0} (\beta)$ when $\beta$ tends to $\alpha$ from below. 
Similarly, if $t + (1-t)U \leq 1 - \varepsilon$, 
we get 
\begin{align*} &  \underset{n \rightarrow \infty}{\lim \sup} \, q_{\nu^{(n)}} \left( \min(1, t + (1-t)U + 3/n)   \right) \\ & \leq \underset{n \rightarrow \infty}{\lim \sup} \, q_{\nu^{(n)}} \left(   t + (1-t)U + \varepsilon/2   \right) \leq q_{\nu_0}  \left( t + (1-t)U + \varepsilon \right)
+ \varepsilon/2,
\end{align*}
and then, letting $\varepsilon \rightarrow 0$, 
we deduce that for $U <1$, and then almost surely, 
$$ \underset{n \rightarrow \infty}{\lim \sup} \, q_{\nu^{(n)}} \left( \min(1, t + (1-t)U + 3/n)   \right) \leq q_{\nu_0} \left( ( t + (1-t)U)+  \right),$$
where $q_{\nu_0}(\alpha+)$ is the limit of $q_{\nu_0} (\beta)$ when $\beta$ tends to $\alpha$ from above. 
Now, $q_{\nu_0}$ is nondecreasing and then has 
at most countably many discontinuities, which implies that almost surely, 
$$  q_{\nu_0} \left( ( t + (1-t)U)+  \right)
=  q_{\nu_0} \left( ( t + (1-t)U)-  \right)
=  q_{\nu_0} \left(  t + (1-t)U  \right).$$
Hence, almost surely,
\begin{align*} 
 q_{\nu_0} & \left(  t + (1-t)U  \right)
  =  q_{\nu_0} \left( ( t + (1-t)U)-  \right)
 \leq \underset{n \rightarrow \infty}{\lim \inf} \, 
q_{\nu^{(n)}} \left( t + (1-t)U - 3/n   \right)
\\ & \leq \underset{n \rightarrow \infty}{\lim \inf} \, q_{\nu^{(n)}} \left( \frac{ \ell + 1 + \lfloor (n-\ell) U \rfloor}{n} \right)
  \leq \underset{n \rightarrow \infty}{\lim \sup} \, q_{\nu^{(n)}} \left( \frac{ \ell + 1 + \lfloor (n-\ell) U \rfloor}{n} \right)
\\ & \leq \underset{n \rightarrow \infty}{\lim \sup} \,
q_{\nu^{(n)}} \left( \min(1, t + (1-t)U + 3/n)   \right)   \leq q_{\nu_0} \left( ( t + (1-t)U)+  \right) 
=  q_{\nu_0} \left(  t + (1-t)U  \right).
\end{align*}
Since $\ell = nt + \mathcal{O}(1)$, we also have 
$$\frac{\lfloor (n-\ell) U \rfloor}{\lfloor (n-\ell) U \rfloor + nt} \underset{n \rightarrow 
\infty}{\longrightarrow}  \frac{ (1-t)U}{t + (1-t)U},$$ 
and then 
$$ q_{\nu^{(n)}} \left( \frac{ \ell + 1 + \lfloor (n-\ell) U \rfloor}{n} \right) \frac{\lfloor (n-\ell) U \rfloor}{\lfloor (n-\ell) U \rfloor + nt}
\underset{n \rightarrow 
\infty}{\longrightarrow} 
 q_{\nu_0} \left(  t + (1-t)U  \right) \frac{ (1-t)U}{t + (1-t)U}
$$
almost surely. Letting 
$V_t := t + (1-t) U$, which is uniformly distributed 
on $[t,1]$, 
this proves convergence of 
the empirical measure of the roots of the $\lfloor n m_n t \rfloor$-th derivative of  
$P_{n,m_n}$ towards $\nu_t \otimes unif$, 
where $\nu_t$ is given by \eqref{nuV}. 

Since $x \mapsto (1-t/x) q_{\nu_0}(x)$ is nondecreasing on $[t,1]$, and 
strictly increasing on  
$[\max(t, \nu_0(\{0\})), 1]$,  
we get, for $x \in [\max(t, \nu_0(\{0\})), 1]$, 
\begin{equation}\mathbb{\nu}_t([0,(1-t/x) q_{\nu_0}(x)])
= \mathbb{P} [ (1 - t/V_t) q_{\nu_0}(V_t) 
\leq (1-t/x) q_{\nu_0}(x) ]
= \mathbb{P} [V_t \leq x] = (x-t)/(1-t),
\label{xvnu}
\end{equation}
which implies 
$$\left( 1 - \frac{t}{y} \right) 
q_{\nu_0}(y) \leq q_{\nu_t} \left( \frac{x-t}{1-t} \right)
\leq \left( 1 - \frac{t}{x} \right) 
q_{\nu_0}(x)
$$
for $\max(t, \nu_0(\{0\})) \leq y < x \leq 1$, 
and then 
$$ \left( 1 - \frac{t}{x} \right) 
q_{\nu_0}(x-) \leq q_{\nu_t} \left( \frac{x-t}{1-t} \right)
\leq \left( 1 - \frac{t}{x} \right) 
q_{\nu_0}(x)$$
for $\max(t, \nu_0(\{0\})) < x \leq 1$. 
Now, for $0 < y < x \leq 1$, 
$$\nu_0( [0, q_{\nu_0}(x-)]) 
\geq \nu_0( [0, q_{\nu_0}(y)]) \geq y,$$
and then 
$$\nu_0( [0, q_{\nu_0}(x-)]) \geq x,$$
which shows that $q_{\nu_0}$ is left-continuous on $(0,1]$, and then 
\begin{equation} q_{\nu_t} \left( \frac{x-t}{1-t} \right) 
= \left(1 - \frac{t}{x} \right) q_{\nu_0}(x)
\label{quantiles}
\end{equation}
for $x \in (\max(t, \nu_0(\{0\})), 1] $. 
If $x = \max(t, \nu_0(\{0\}))$, 
we still have \eqref{xvnu}, which implies that the left-hand side of \eqref{quantiles} is at most the right-hand side. Since the right-hand side is zero in this case, \eqref{quantiles} remains true. 
Since the quantile functions are nondecreasing, 
the two sides of \eqref{quantiles} are at most zero 
if $t \leq x <  \max(t, \nu_0(\{0\}))$, which implies that they are both equal to zero. Hence, 
\eqref{quantiles} is true for all $x \in [t, 1]$, 
which proves \eqref{eqquantiles}.

In the case where the distribution function 
of $\nu_0$ is a continuous and strictly increasing bijection from $[0,1]$ to  $[0,A]$ for some $A > 0$,
it has a reverse, continuous and strictly increasing bijection, which coincides with $q_{\nu_0}$. By \eqref{eqquantiles}, 
$q_{(1-t)\nu_t}$ is a continuous, strictly increasing bijection from $[0, 1-t]$ to $[0, A(1-t)]$, which implies that the distribution 
function of $(1-t) \nu_t$ is a continuous and  strictly increasing bijection from $[0, A(1-t)]$ to $[0, 1-t]$. We then deduce \eqref{Psi} from \eqref{eqquantiles}.

If $\nu_0$ is absolutely continuous with respect to the Lebesgue measure, supported on $[0,A]$ for some $A > 0$, with a continuous, positive density $\psi_0$ 
on $(0, A)$, it follows that its distribution function
$\Psi_0$ is an increasing bijection from $[0,A]$
to $[0,1]$. The reverse bijection, from $[0,1]$ to $[0,A]$, is $q_{\nu_0}$: this function is continuously differentiable on $(0,1)$,
with 
$$q_{\nu_0}'(\alpha) = \frac{1}{ \psi_0(q_{\nu_0}(\alpha))}$$
for $\alpha \in (0,1)$: notice that $\psi_0$ does not vanish on $(0,A)$ by assumption. 
Since for $t \in [0,1)$, $\alpha \in [0, 1-t]$, 
\begin{equation} q_t (\alpha) := q_{(1-t)\nu_t} (\alpha) = \frac{\alpha}{\alpha+t} q_0 (\alpha + t), \label{qtalpha}
\end{equation}
we deduce that $q_{t}$ is continuous, 
strictly increasing  on $[0, 1-t]$, and continuously differentiable, with strictly positive derivative, on $(0, 1-t)$: moreover, it extends, by \eqref{qtalpha}, to a continuously differentiable 
function of two variables $t$ and $\alpha$, on the domain $\alpha, t \in \mathbb{R}$, $0 < \alpha + t < 1$. 
Direct 
computation provides, on this domain, the equation:
\begin{align}\label{e:inv_Psit}
\frac{\partial}{\partial \alpha} 
q_{t} (\alpha)
= \frac{\partial}{\partial t} 
q_{t} (\alpha) + \frac{q_t (\alpha) }{\alpha}.  
\end{align}
The properties of $q_{t}$ above imply that this function is an increasing bijection 
from $[0, 1-t]$ to the interval $[0, A(1-t)]$, 
 its reverse bijection $\Psi_t$, which is the distribution function of $(1-t) \nu_t$, 
being differentiable on $(0, A(1-t))$ with strictly positive derivative. This derivative 
is a density $x \mapsto \psi(x,t)$ for the 
distribution of $(1-t) \nu_t$. 

Moreover, for $t \in [0,1)$, $x \in (0, A(1-t))$, $\Psi_t (x)$ is the unique $\alpha \in (0, 1-t)$ such that $f(t,x, \alpha)= 0$, 
where 
$$f(t,x, \alpha) := q_{t}(\alpha) - x.$$
Again using \eqref{qtalpha}, $f$ extends 
to a continuously differentiable function 
on the domain $t, x, \alpha \in \mathbb{R}$, 
$0 < t + \alpha < 1$. Since the derivative of $f$
with respect to $\alpha$ is strictly positive 
at $(t, x, \alpha)$ when $t \geq 0$, 
$\alpha \in (0, 1-t)$, implicit function 
theorem shows that the two variable function 
$(t, x) \mapsto \Psi_t (x)$ is continuously 
differentiable on the set
$$\{(t,x) \in [0,1) \times \mathbb{R}, 
x \in (0, A(1-t))\}.$$
We can then differentiate the identity $\Psi_t(q_t(\alpha)) = \alpha$ with respect to $\alpha$ and get, for $t \in [0, 1)$, $\alpha \in (0, 1-t)$, 
\begin{align*}
    \frac{\partial}{\partial \alpha} \Psi_t(q_t(\alpha)) = \psi(q_t(\alpha), t) \cdot \frac{\partial q_t(\alpha)}{\partial \alpha} = 1, 
\end{align*}
which implies 
\begin{align*}
\frac{\partial q_t(\alpha)}{\partial \alpha} = \frac{1}{\psi(q_t(\alpha), t)}.
\end{align*}
Differentiating $\Psi_t(q_t(\alpha)) = \alpha$ with respect to $t$ yields:
\begin{align*}
\frac{\partial \Psi_t}{\partial t}(q_t(\alpha)) + \psi(q_t(\alpha), t) \cdot \frac{\partial q_t(\alpha)}{\partial t} = 0,
\end{align*}
equivalently,
\begin{align*}
\frac{\partial \Psi_t}{\partial t}(q_t(\alpha)) = -\psi(q_t(\alpha), t) \cdot \frac{\partial q_t(\alpha)}{\partial t}.
\end{align*}
Substituting the expression for $\frac{\partial q_t(\alpha)}{\partial t}$ from \eqref{e:inv_Psit}, we obtain
\begin{align*}
\frac{\partial q_t(\alpha)}{\partial t} 
= \frac{\partial q_t(\alpha)}{\partial \alpha} -\frac{q_t(\alpha)}{\alpha} 
= \frac{1}{\psi(q_t(\alpha), t)} - \frac{q_t(\alpha)}{\alpha}.
\end{align*}
Thus,

\begin{align*}
\frac{\partial \Psi_t}{\partial t}(q_t(\alpha)) 
= -\psi(q_t(\alpha), t) \cdot 
\left( \frac{1}{\psi(q_t(\alpha), t)} - \frac{q_t(\alpha)}{\alpha} \right)
= -1 + \psi(q_t(\alpha), t) \cdot \frac{q_t(\alpha)}{\alpha}.
\end{align*}
Set $x = q_t(\alpha)$, so $\alpha = \Psi_t(x)$. We get, for $t \in [0, 1)$, $x \in (0, A(1-t))$, 

\begin{align}
\frac{\partial \Psi_t}{\partial t}(x) = -1 + \frac{\psi(x,t) \cdot x}{\Psi_t(x)},
\label{diffequationPsi}
\end{align}
which is equivalent to \eqref{equationPsixt}. 

Now, let us make the extra assumption that the density $\psi_0$ of $\nu_0$ is continuously differentiable on 
$(0,A)$. In this case, $\Psi_0$ is
twice continuously differentiable on $(0,A)$, 
and then $q_{\nu_0}$ is twice continuously differentiable on $(0,1)$, which implies 
that \eqref{qtalpha} defines 
a twice continuously differentiable function of 
two variables $\alpha, t \in \mathbb{R}$, 
$0 < \alpha + t < 1$. 
Applying the implicit function theorem 
as above implies that 
$(t, x) \mapsto \Psi_t (x)$ is twice continuously differentiable on 
$$\{(t,x) \in [0,1) \times \mathbb{R}, 
x \in (0, A(1-t))\},$$
and that the density $\psi$ is continuously differentiable on 
the same set. 
We can then differentiate \eqref{diffequationPsi} with respect to 
$x$, 
and get 
$$\frac{\partial^2 \Psi_t (x)}{\partial t \partial x} = \frac{\partial}{\partial x} 
\left( \frac{x \psi(x,t)}{\Psi_t (x)} \right).$$
Equivalently, 
$$
\frac{\partial \psi}{\partial t}(x,t) 
= \frac{\partial}{\partial x} \left( 
\frac{\psi(x,t)}{\frac{1}{x} \int_0^x \psi(y,t) dy} 
\right),$$
which is the PDE \eqref{PDErotationalmaintheorem}.
 \section{Discussion}\label{Discussion}

Our work is related to the precise conjecture stated by Hoskins and Kabluchko in \cite{hoskins2021dynamics}, as recalled in the introduction, inspired by O'Rourke and Steinerberger \cite{o2021nonlocal} based on a mean field approach amenable to a hydrodynamic approximation of the considered dynamics of root sets. We emphasized the role of the sampling of the initial rotationally invariant measure for defining the motion of the roots under differentiation, since the limit, if it exists, can depend on the choice of the sampling. For example, if $\nu_0$ is Dirac measure at $1$ in Theorem \ref{main}, 
we find that $\nu_t$ has density 
$$y \mapsto \frac{t}{(1-t)(1-y)^2}$$ on the interval $[0, 1-t]$, whereas a sampling with roots of unity 
gives all roots of iterated derivatives equal to zero. 
We have chosen the setting of Theorem \ref{main} in such a way that a large part of its study  reduced to a  one-dimensional dynamics with a modified differentiation operator.

We have proven the equivalent of the conjecture in \cite{hoskins2021dynamics} for specific samplings of rotationally invariant measures (defined 
with a pair of growing integers $(n,m)$ in Section \ref{3}) under a quite mild assumption on the 
relative growth between $n$ and $m$, namely 
$m/(n \log n)$ going to infinity with $n$. We expect that a more precise analysis of the setting 
can be done in order to relax this assumption. 
In particular, we conjecture that the conclusion of Theorem \ref{main} remains true for $m$ with the same order of magnitude as $n$, and maybe even under weaker assumption.
To prove that, we could take into account in our estimates of the sum \eqref{rootsS}, computed in Lemma \ref{upperbound}, that the values of the 
terms  corresponding to radii bigger and smaller than $z_k$ are of opposed signs and induce a compensation. This feature would allow to improve the bound given statement of Lemma \ref{upperbound} for values of $\alpha$ which are close to one. We leave this direction of improvement for a future work, and in the next section, we give an encouraging example.

We notice that in \cite{hoskins2021dynamics}, Hoskins and Kabluchko proved the conjecture for another class of samplings of the roots sets induced by specific expressions of coefficients of the polynomials. The format of this class of samplings is stable by differentiation, as well as the setting of the present article. 

The precise statement of the conjecture in \cite{hoskins2021dynamics} requires that the sampling of the initial rotationally invariant measure is i.i.d.: this is a strong assumption which is not stable by differentiation, hence we expect that a proof of the conjecture in this setting should be extended to more general samplings of the initial rotationally invariant distribution. It is worth noticing that in the example provided in \cite{hoskins2021dynamics} the root sets accumulated rather in a finite number of rings than on a finite number of circles. The stability of our model with points on circles can be related to the lack of noise in the distribution of the points on the circles, as shown by examples in the next section.

The article \cite{campbell2024fractional}, by Campbell, O'Rourke and Renfrew, presents an
interesting connection between fractional free convolution of Brown measures of R-diagonal operators and motion under differentiations of root sets of polynomials issued from a rotationally invariant measure on $\CC$, illustrated by examples. They give an heuristic, they call a ``formal proof'', reinforcing the conjecture we studied, but it is not a rigorous proof.

\section{Examples and prospective}\label{Examples_prospective}

In each of the following subsections, we provide a prospective question on iterated differentiations of polynomials with complex roots, illustrated with
simulations.
\subsection{The case $m \leq n$}
We consider our model of sample with $m=20$, and two values of $n$, $n=20$ and $n=40$ with an intial distribution of radii equal to $1-\frac{j^2}{n^2}$, then we compare the  radii corresponding to the polynomial after 100 differentiations, with the radii predicted by Hoskins and Kabluchko's conjecture.  Figure \ref{m=20,n=20,40with100der} shows that in both cases, the prediction is rather satisfied. It would be good to relax the assumption on $m_n$ in Theorem \ref{main} in order to cover some cases where $m \leq n$.
\begin{figure}[ht!]
  \includegraphics[width=5cm]{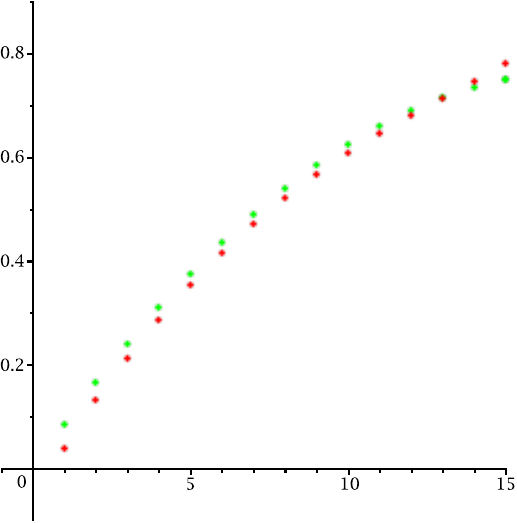}
   \includegraphics[width=5cm]{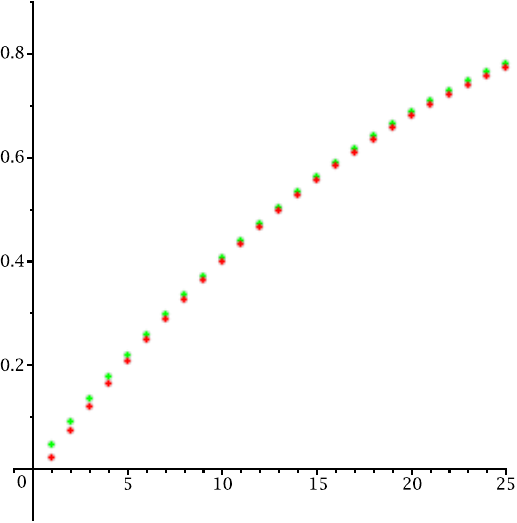}
   \caption{$m=20$, with $n=20$ at left and $n=40$ at right after 100 differentiations.}\label{m=20,n=20,40with100der}
\end{figure}    

\subsection{Very small perturbations}
We consider small angular perturbations of order $1/m$, on our model of sample with $m=n=10$, Figure \ref{m=n=10with0and30der} shows the initialy perturbed sample and the effect of 30 differentiations.
\begin{figure}[ht!]
  \includegraphics[width=5cm]{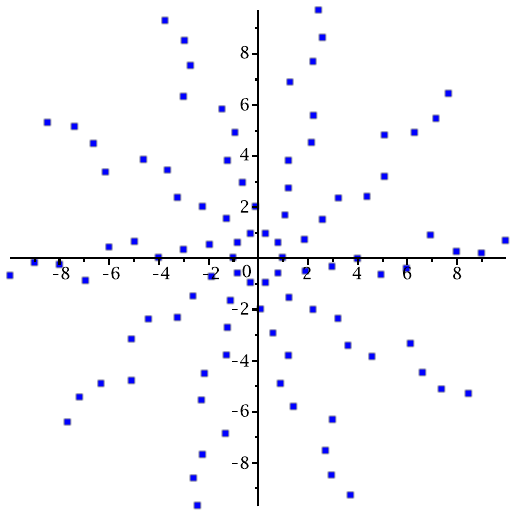}
   \includegraphics[width=5cm]{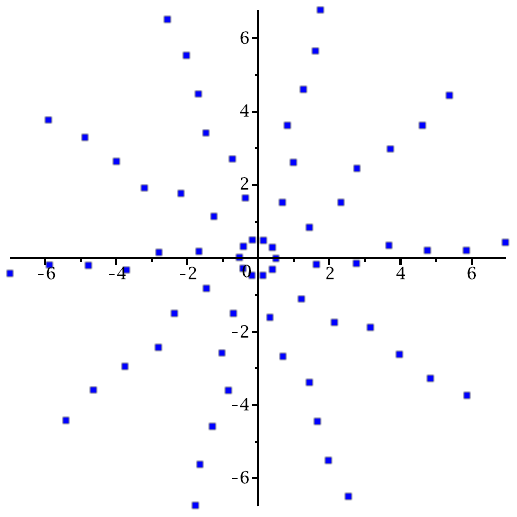}
   \caption{$m=10$, $n=10$,  0 and 30 differentiations on an initially perturbed sample.}\label{m=n=10with0and30der}
 \end{figure}   
   We notice a regularization effect of the differentiations.

\subsection{Uniform distribution on each circle}
In this subsection, we sample the uniform distribution on $n$ circles by  $n$ sets of   $m$ i.i.d. uniform random variables on each circle, so in total $mn$ random variables taking values in $[0,2\pi)$ for the argument of the points. The left of Figure \ref{m=n=20iidud} illustrates the shape of the initial root set, 
distributed on 20 concentric circles. The right of Figure \ref{m=n=20iidud} illustrates the shape of the root sets after applying 100 differentiations and then 200 differentiations. We see that the second and third  root sets are no more distributed on families of circles. We notice the decrease of the radii of the discs containing the root sets and the appearance of filaments going towards the origin.

This behavior can be compared with the one illustrated in the article  \cite{hoskins2021dynamics} for polynomials defined from their coefficients.

\begin{figure}[ht!]
  \includegraphics[width=5cm]{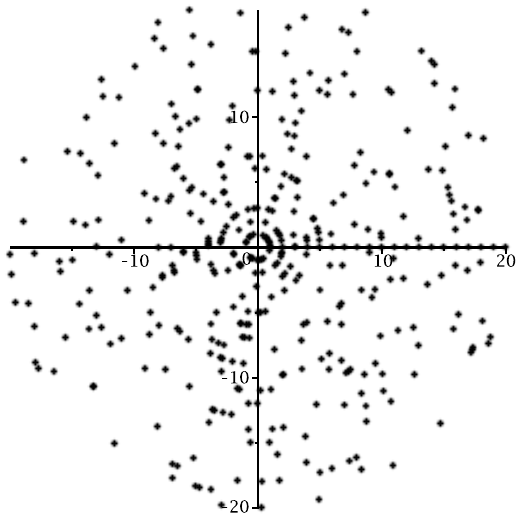}
   \includegraphics[width=5cm]{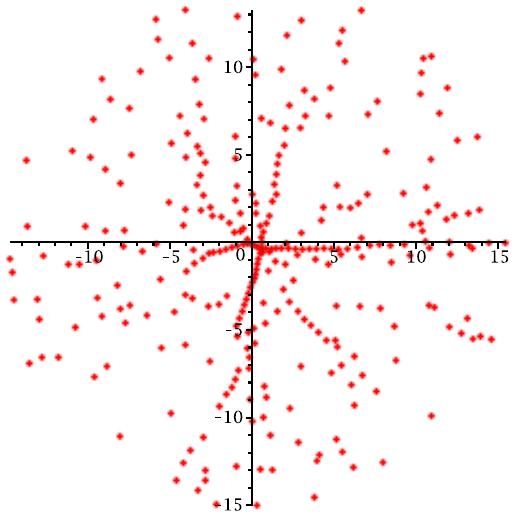}
   \includegraphics[width=5cm]{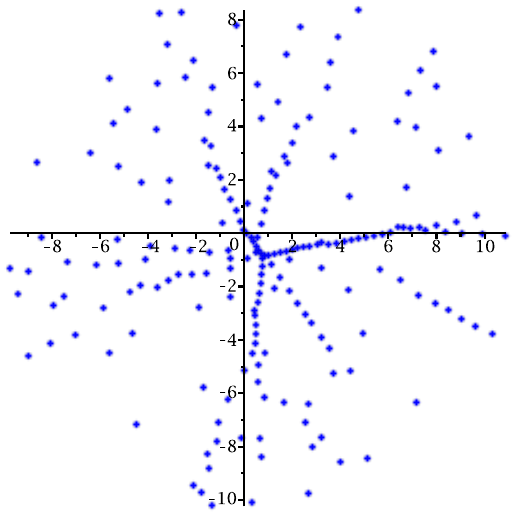}
   \caption{$m=20, n=20$, i.i.d. uniform distribution of the arguments.}\label{m=n=20iidud}
 \end{figure}

%\subsection{Large mass of $\mu_0$ near $1$ and small $m$}

\subsection{Circular derivatives}
One can also investigate the behavior of the root set of a polynomial of the form $\prod_{j=1}^n (z^m-r_{j}^m)$ , under the iterative actions of the circular derivative operator $\mathcal{D}_N= z\frac{d}{dz}-\frac{N}{2}$ which conserves the degree $N := mn$.

Again, we have a dynamics which diminishes the radii of the  more ``external'' circles and increases the radii of the  more ``internal''  circles, hence asymptotically concentrates the $N$ points on the circle of radius equal to the geometric mean $R$ of the radii, which is conserved by the operator $\mathcal{D}_N$. We illustrate this behavior with the following Figure \ref{0-2-30-200} showing the  configuration at successive times of its evolution, for $m=15, n=12$, $r_j = j$ for $1 \leq j \leq 12$, 
$R = (12!)^{1/12} \simeq 5.29$.

\begin{figure}[ht!] 
  \includegraphics[width=5cm]{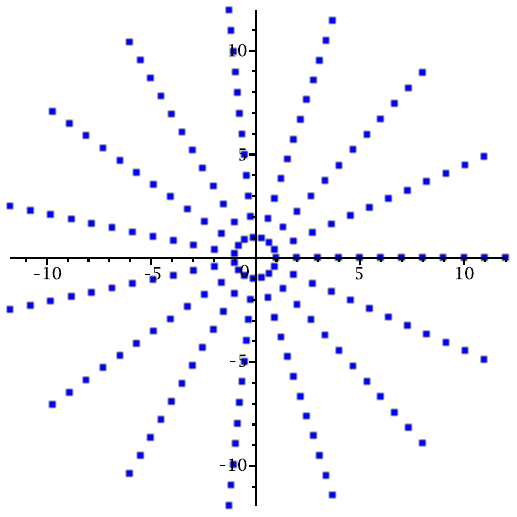}
   \includegraphics[width=5cm]{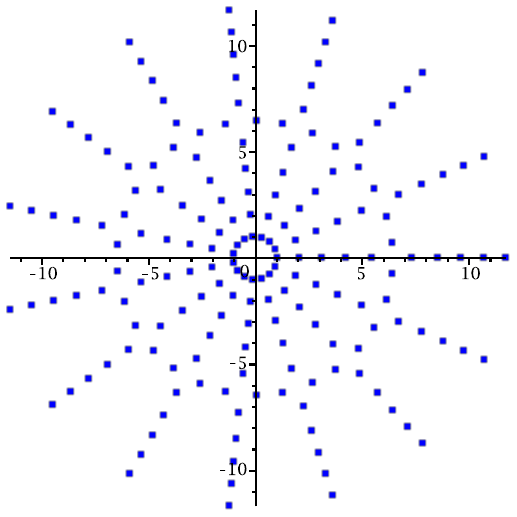}
   \includegraphics[width=5cm]{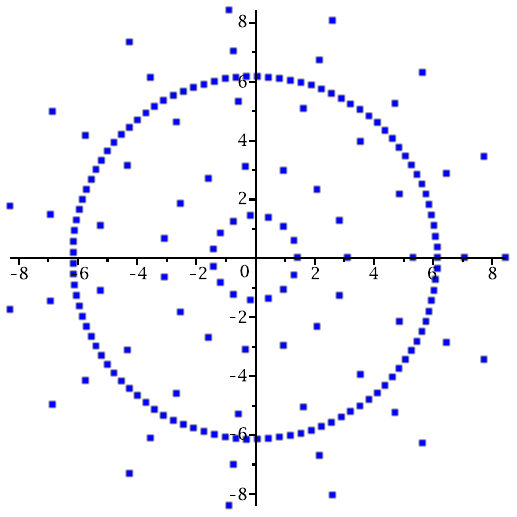}
     \includegraphics[width=5cm]{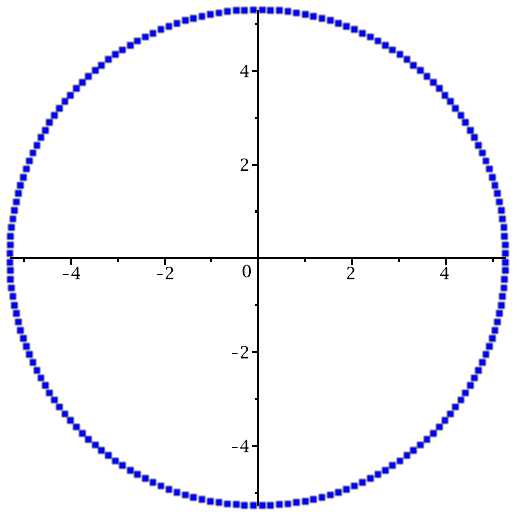}
   \caption{Action of 0, 2, 30, and 200  circular differentiations.} \label{0-2-30-200}
\end{figure}  

We notice an intricate dynamics. This is due to the fact that even if at each differentiation we have a polynomial of the form $ \prod_{j=1}^N (z^m-z_{j}^m)$ multiplied by a constant, in this setting these $z_j$ may become real negative and also non real numbers. Hence the arguments of their  $m$-th roots equal to $\frac{2 \pi  k}{m} +\theta_j$ for $0 \leq k \leq m-1$ and some $\theta_j$ as seen on the pictures. Moreover, the different ``speeds'' of the radii $r_j$ towards $R$ may induce some early collisions and the creation of circles with more than $m$ points. 

\subsection*{Acknowledgments}
AG thanks the European ERC 101054746 grant ELISA for its support (indirect costs).
%%%  
%\bibliographystyle{abbrv}
%\bibliography{GNV2}

\end{document}